\documentclass[11pt,twoside,a4paper]{amsart}

\usepackage{amsfonts,amsmath,amsthm}
\usepackage{hyperref}

\newtheorem{thm}{Theorem}[section]
\newtheorem{lem}[thm]{Lemma}
\newtheorem{pro}[thm]{Proposition}
\newtheorem{cor}[thm]{Corollary}

\theoremstyle{remark}
\newtheorem{remark}[thm]{Remark}
\theoremstyle{definition}
\newtheorem{definition}[thm]{Definition}

\newcommand{\R}{\mathbb{R}}
\newcommand{\C}{\mathbb{C}}
\newcommand{\I}{\mathbb{I}}
\newcommand{\Z}{\mathbb{Z}}
\newcommand{\Rd}{\mathbb{R}^{3}}
\newcommand{\A}{\mathcal{A}}
\newcommand{\Cc}{\mathcal{C}_c^{\infty}(\Rd,\C^2)}
\newcommand{\Ccc}{\mathcal{C}_c^{\infty}(\Rd,\C^4)}

\numberwithin{equation}{section}

\begin{document}

\title[Self-adjoint extensions of Dirac operators]{Distinguished self-adjoint extensions of Dirac operators via Hardy-Dirac inequalities}

\author{Naiara Arrizabalaga}

\address{N. Arrizabalaga: Universidad del Pa\'is Vasco, Departamento de Matem\'aticas, Apartado 644, 48080, Spain}
\email{naiara.arrizabalaga@ehu.es}

\thanks{The author is supported by the grant PIFA-B/01/06.}

\begin{abstract} We prove some Hardy-Dirac inequalities with two different weights including measure valued and Coulombic ones. Those inequalities are used to construct distinguished self-adjoint extensions  of Dirac operators for a class of diagonal potentials related to the weights in the above mentioned inequalities.
\end{abstract}


\maketitle

\section{Introduction}\label{Intro}

In this work we deal with the problem of self-adjointness of Dirac operators. Many authors have studied this problem for Dirac operators $H_0$ coupled to an electrostatic potential $V$. Denoting $H_0 = -i \alpha\cdot\nabla + \beta$ where $\alpha = (\alpha_1,\alpha_2,\alpha_3)$,
$$\alpha_i:= \left( \begin{array}{rr} 0 & \sigma_i  \\ \sigma_i & 0  \end{array} \right)\ , \ i=1,2,3 \quad, \ \ \beta:=\left( \begin{array}{rr} \I_2 & 0  \\ 0 & -\I_2  \end{array} \right),$$
$\I_2$ is the identity operator on $\C^2$ and 
$$\sigma_1= \left( \begin{array}{rr} 0 & 1  \\ 1 & 0  \end{array} \right),\; \sigma_2= \left( \begin{array}{rr} 0 & -i  \\ i & 0  \end{array} \right),\; \sigma_3= \left( \begin{array}{rr} 1 & 0  \\ 0 & -1  \end{array} \right)$$
is the family of Pauli matrices.

If $V$ is a bounded function which tends to 0 at infinity, the operator $H_0+V$ with domain $H^1(\Rd, \C^4)$ is self-adjoint, see for instance \cite{T}. However, if $V$ has singularities, one is interested in constructing self-adjoint extensions of $H_0+V$ originally defined on the domain $\Ccc$. The papers  \cite{S2,S1,W1,W2,W3,N,KW,EL} treated this problem by using a different method depending on the singularity of the potential. Those works dealt exclusively with electrostatic potentials, while in \cite{A1,A2,A3,Y}, Arai and Yamada consider more general matrix-valued potentials. Results on the essential self-adjointness of Dirac operators with relativistic $\delta$-sphere interactions can be found in \cite{D1,D2,DA} and similar results for the Schr\"odinger operator with point interactions in \cite{AGHH}. See notes in \cite{T} for the complete bibliography.

We will restrict our attention to \cite{EL}, the most recent work among the above mentioned ones, in which Esteban and Loss use a method based on Hardy-like inequalities. Let us explain this result in more detail. Let $V:\Rd\rightarrow \R$ be a potential such that for some constant $c(V)\in(-1,1)$, $\Gamma:=\sup_{\Rd} V < 1+c(V)$ and for every $\varphi\in\Cc$,
\begin{equation}\label{ineq.EL}
\int_{\Rd}\left(\frac{|\sigma\cdot\nabla\varphi|^2}{1+c(V)-V}+(1-c(V)+V)|\varphi|^2\right)\, dx \geq 0.
\end{equation}
Then Esteban and Loss can construct a distinguished self-adjoint extension of the operator $H_0+V$ defined on $\Ccc$.  One of the components of the operator is extended by using the Friedrichs extension and inequality (\ref{ineq.EL}), and the remaining one by choosing the right domain for the whole operator. In \cite{ELAbs}, the same authors point out that an extra condition on the potential is needed for the construction of the self-adjoint extensions mentioned in \cite{EL}. The natural condition to get the desired symmetry on the operator $H_0+V$ is that each component of 
\begin{equation}\label{cond}
(\gamma-V)^{-2}\nabla V
\end{equation}
is square integrable, where $\gamma$ is any number in $(\Gamma, 1+c(V))$.

In this paper we generalize, in some sense, the above mentioned result. Consider $H_0=-i \alpha\cdot\nabla + m \beta$, using similar techniques as in \cite{EL}, we construct distinguished self-adjoint extensions of Dirac operators defined as $H_V= H_0-V$ with potentials of the type
$$V(x)= \left( \begin{array}{rr} w_1(x)\I_2 &  0\quad  \\  0\quad  & w_2(x)\I_2  \end{array} \right)$$ where $w_1$ is a real function or a singular measure and $w_2$ is a function. 

Assuming that $w_2$ is positive, for $w_1$ negative the proof runs quite straightforward.  However, for the positive sign of $w_1$ we need to prove Hardy-Dirac inequalities such as 
\begin{equation}\label{first ineq.}
\int_{\Rd}w_1|\phi|^2 \leq \int_{\Rd}\frac{|\sigma \cdot \nabla\phi|^2}{m+ w_2-\lambda}+(m+\lambda)\int_{\Rd}|\phi|^2,
\end{equation}
for some $\lambda\in(-m,m)$. 

Estimates of the type (\ref{first ineq.}) are proved in Section \ref{sec:Hardy} and we use them in Section \ref{sec:Self-ad} to prove the self-adjointness of the Dirac operator $H_V$. In particular, they are used to define a Hilbert space $\mathcal{H}$ with the inner product 
\begin{equation*}
(\phi,\varphi)_{\mathcal{H}} := \int_{\Rd}(m-w_1+\lambda)\phi \cdot \overline{\varphi}  + \int_{\Rd} \frac{i \sigma \cdot \nabla\phi}{m+w_2-\lambda} \cdot \overline{ i \sigma \cdot \nabla\varphi}.
\end{equation*}
By using the Riesz Representation Theorem we are able to extend one component of the operator. The remaining component is extended by choosing the right domain $\mathcal{D}$. Moreover, we avoid the extra condition on the gradient of the potential equivalent to (\ref{cond}) thanks to the particular structure of the inner product, which is itself symmetric. Therefore, if we take $w_1=w_2=V$ such that $\sup_{x\neq0}V(x)\leq \frac{\nu}{|x|}$, we improve the result of \cite{EL} in the sense that we can construct a distinguished self-adjoint extensions without using the condition (\ref{cond}).

Some examples of Dirac operators and Hardy-Dirac inequalities are given in the last section of the paper.

\section{Hardy-Dirac estimate}\label{sec:Hardy}

\begin{definition}\label{def. A} 
Let  $\A$ be  the class of potentials that contains all pairs of  positive radial measurable functions, $V_1, V_2: \Rd\to\R^+$, that satisfy

$$A_+[V_1,V_2] := \sup_{r>0}\left[\frac{1}{r^2}\int_0^r (V_1(t)+V_2(t))\ t^2\ dt\right] < +\infty$$ and $$A_-[V_1,V_2] := \sup_{r>0}\left[r^2\int_r^{\infty} (V_1(t)+V_2(t))\ \frac{dt}{t^2}\right]< +\infty.$$
\end{definition}

We can now state the main result of this section.
\begin{thm}\label{thm:Hardyv1v2}
Let $V_1,V_2\in \A$. For any $\phi\in L^2(\Rd,\C^2)$ and any $\gamma \geq 0$,
\begin{equation*}
\int_{\Rd}V_1|\phi|^2 \leq \max\{A^2_+,A^2_-\}\int_{\Rd}\frac{|\sigma \cdot \nabla\phi|^2}{V_2+\gamma}+\gamma\int_{\Rd}|\phi|^2.
\end{equation*}
\end{thm}

The inequality holds whenever the right hand side is finite. We follow the approach of \cite{DDEV} for the proof. For the convenience of the reader we state the relevant results on the spectrum of $\sigma\cdot L$ and the projections associated to the spectral space, $X_k$, without proofs, thus making our exposition self-contained.

\begin{lem}\label{lem:spec}
The spectrum of $\sigma\cdot L$ is the discrete set $\{k\in\Z : k \neq -1\}$ and $\sigma\cdot L$ applied to a radial function is zero. Moreover, if $\phi$ is a continuous function, then $P_k\phi(0)=0$ for any $k\in\Z / \{0,-1\}$.
\end{lem}

The key points are that L commutes with all radial functions and that -1 is not in the spectrum of $\sigma\cdot L$.

\begin{lem}
For any $k,l \in Spec(\sigma\cdot L)$, $k\neq l$, $P_k(\sigma\cdot L)^2P_l\equiv P_l(\sigma\cdot L)^2P_k\equiv 0$ in $H^1(\Rd,\C^2)$.
\end{lem}

\begin{cor}\label{cor:sum_k}
Any function $\phi \in L^2(\Rd,\C^2)$ can be written $$\phi = \sum_{k\in\Z : k \neq -1}\phi_k$$ with $\phi_k\in X_k$ and moreover, if $W$ is a radial function,
\begin{equation*}
\int_{\Rd}W|\phi|^2 = \sum_{k\in\Z : k \neq -1} \int_{\Rd}W|\phi_k|^2,
\end{equation*}
\begin{equation*}
\int_{\Rd}W|\sigma \cdot \nabla\phi|^2= \sum_{k\in\Z : k \neq -1} \int_{\Rd}W|\sigma \cdot \nabla\phi_k|^2.
\end{equation*}
\end{cor}

\begin{definition}
For $V_1,V_2 \in \A$, $A_k$ is given by
\begin{equation*}
A_k := \left\{ \begin{array}{rl} \displaystyle{\sup_{r>0}}\quad \frac{1}{r^{2(k+1)}}\int_0^r (V_1(s)+V_2(s))\ s^{2(k+1)}\ ds ,\;  k\in\Z , k\geq0, \bigskip\\ \displaystyle{\sup_{r>0}}\quad r^{2(k+1)}\int_r^{\infty} (V_1(s)+V_2(s))\ \frac{ds}{s^{2(k+1)}}, \;   k\in\Z , k\leq-2. \end{array} \right.
\end{equation*}
\end{definition}

We can see at once that $A_k\leq A_0$ for all $k\in\Z , k\geq0$ and $A_k\leq A_{-2}$ for all $k\in\Z , k\leq -2$, because $V_1$ and $V_2$ are nonnegative.

\begin{proof}[Proof of Theorem \ref{thm:Hardyv1v2}]
Let $\phi\in\Cc$. From the fundamental theorem of calculus  we write 
$$|\phi(x)|^2 = \Re\left(\int_r^{\infty} -2 \overline{\phi(t\omega)} (\omega\cdot\nabla\phi(t\omega))\, dt\right)$$
for $r=|x|$, $\omega=\frac{x}{|x|}$. Suppose that $W$ is a radial and real function, then

\begin{eqnarray*}
\int_{\Rd}W|\phi|^2 & = & -2\,\Re\left(\int_{S^2}\; d\omega\int_0^{\infty} W(r\omega)r^2\;dr\int_r^{\infty}\overline{\phi(t\omega)} (\omega\cdot\nabla\phi(t\omega))\, dt\right)\\
& = & -2\,\Re\left(\int_{S^2}\; d\omega\int_0^{\infty} \overline{\phi(t\omega)} (\omega\cdot\nabla\phi(t\omega))t^2g_W(t)\;dt\right)
\end{eqnarray*}
where
$$g_W(t):= \frac{1}{t^2}\int_0^t W(r)\,r^2dr.$$
By abuse of notation, we use the same letter $W$ for $W(x)$ and $W(r)$. Now using the identity 
\begin{equation}\label{identity}
\frac{x}{|x|}\cdot\nabla = \left(\sigma\cdot\frac{x}{|x|}\right)\left(\sigma\cdot\nabla\right)+\frac{1}{|x|}\left(\sigma\cdot L \right) \quad \forall x\in\Rd,
\end{equation}
for any $\delta >0$ and $\gamma>0$, we obtain
\begin{equation*}
\left\langle \left(W+\frac{2}{|x|} g_W \sigma\cdot L \right)\phi, \phi\right\rangle  =  -2\,\Re\left(\int_{\Rd} \overline{\phi}\left(\sigma\cdot\frac{x}{|x|}\right)\left(\sigma\cdot\nabla\right) \phi g_W(|x|)\, dx\right)
\end{equation*}
\begin{equation}\label{sep-pesos}
\leq  ||g_W||_{L^\infty}(0,\infty)\left[\frac{1}{\delta}\int_{\Rd} \frac{|\sigma \cdot \nabla\phi|^2}{V_2+\gamma}\, dx+\delta\int_{\Rd}(V_2+\gamma)|\phi|^2\, dx\right].
\end{equation}

Take the nonnegative spectrum of $\sigma\cdot L$, i.e., $k\in\Z , k\geq0$. We want to solve
\begin{equation*}
W_k+\frac{2 k}{r}g_k = V_1+V_2 \quad \forall r\in (0,\infty)
\end{equation*}
where
\begin{equation*}
g_k= g_{W_k}(r):= \frac{1}{r^2}\int_0^r W_k(s)\,s^2ds.
\end{equation*}
Since
\begin{eqnarray*}
\frac{d}{dr}\left(r^{2k}\int_0^r s^2 W_k(s)\; ds\right) &=& r^{2(k+1)}W_k+2k\,r^{2k-1}\int_0^r s^2 W_k(s)\; ds\\ &=& r^{2(k+1)}(V_1+V_2),
\end{eqnarray*}
then
\begin{equation*}
r^{2k}\int_0^r s^2 W_k(s)\; ds= \int_0^r s^{2(k+1)}(V_1(s)+V_2(s)) \; ds.
\end{equation*}
The equation is solved by 
\begin{equation*}
g_k(r):= \frac{1}{r^{2(k+1)}}\int_0^r (V_1(s)+V_2(s))s^{2(k+1)}ds
\end{equation*}
and
\begin{equation*}
W_k=V_1+V_2-\frac{2k}{r}g_k.
\end{equation*}
\\By definition of $A_k$, $||g_W||_{L^\infty(0,\infty)} = A_k$. From the above and (\ref{sep-pesos}) it follows that for $\phi=\phi_k\in\Cc$ such that 
$$\sigma\cdot L \phi_k= k \phi_k, \quad k\in\Z , k\geq0 ,$$
\begin{equation*}
\int_{\Rd} V_1|\phi_k|^2 + \int_{\Rd} V_2|\phi_k|^2  \leq \frac{A_k}{\delta}\int_{\Rd} \frac{|\sigma \cdot \nabla\phi_k|^2}{V_2+\gamma}+ \delta A_k\int_{\Rd}(V_2+\gamma)|\phi_k|^2.
\end{equation*}
\\Take $\displaystyle{\delta = \frac{1}{A_k}}$,
\begin{equation*}
\int_{\Rd} V_1|\phi_k|^2\leq A^2_k \int_{\Rd} \frac{|\sigma \cdot \nabla\phi_k|^2}{V_2+\gamma}+\gamma\int_{\Rd}|\phi_k|^2.
\end{equation*}
Since $A_k\leq A_0 = A_+$ for all $k\in\Z , k\geq0$, 
\begin{equation}\label{ineq+}
\int_{\Rd} V_1|\phi_k|^2\leq A^2_+ \int_{\Rd} \frac{|\sigma \cdot \nabla\phi_k|^2}{V_2+\gamma}+\gamma\int_{\Rd}|\phi_k|^2.
\end{equation}

We now apply the same argument for the negative spectrum of $\sigma\cdot L$. Let  $\phi\in\Cc$ and write

$$|\phi(x)|^2 = |\phi(0)|^2 + 2\, \Re\left(\int_0^r  \overline{\phi(t\omega)} (\omega\cdot\nabla\phi(t\omega))\, dt\right)$$
for $r=|x|$, $\omega=\frac{x}{|x|}$. Using the same notation as before and assuming that $\phi(0)=0$,

\begin{eqnarray*}
\int_{\Rd}W|\phi|^2 & = & 2\, \Re\left(\int_{S^2}\; d\omega\int_0^{\infty} \overline{\phi(t\omega)} (\omega\cdot\nabla\phi(t\omega))t^2h_W(t)\;dt\right)
\end{eqnarray*}
where
$$h_W(t):= \frac{1}{t^2}\int_t^{\infty} W(r)\, r^2dr.$$

By (\ref{identity}) and for any $\delta >0$ and $\gamma>0$, we obtain
\begin{equation*}
\left\langle \left(W-\frac{2}{|x|} h_W \sigma\cdot L \right)\phi, \phi\right\rangle 
\end{equation*}
\begin{equation}\label{sep-pesos-neg}
 \leq  ||h_W||_{L^\infty(0,\infty)}\left[\frac{1}{\delta}\int_{\Rd} \frac{|\sigma \cdot \nabla\phi|^2}{V_2+\gamma}\, dx+\delta\int_{\Rd}(V_2+\gamma)|\phi|^2\, dx\right].
\end{equation}
\\For the negative spectrum of $\sigma\cdot L$, i.e., $k\in\Z , k\leq-2$ we solve
\begin{equation}\label{ec:W}
W_k-\frac{2 k}{r}h_k = V_1+V_2 \quad \forall r\in (0,\infty)
\end{equation}
where
\begin{equation*}
h_k= h_{W_k}(r):= \frac{1}{r^2}\int_r^{\infty} W_k(s)s^2ds.
\end{equation*}
Since
\begin{eqnarray*}
\frac{d}{dr}\left(r^{2k}\int_r^{\infty} s^2 W_k(s)\; ds\right) & = & r^{2(k+1)}W_k+2k\,r^{2k-1}\int_r^{\infty} s^2 W_k(s)\; ds \\ & = & -r^{2(k+1)}(V_1+V_2)\; .
\end{eqnarray*}
Equation (\ref{ec:W}) can be solved by taking
\begin{equation*}
h_k(r):= \frac{1}{r^{2(k+1)}}\int_r^{\infty} (V_1(s)+V_2(s))s^{2(k+1)}ds
\end{equation*}
and
\begin{equation*}
W_k=V_1+V_2+\frac{2k}{r}h_k \;.
\end{equation*}
\\By definition $||h_W||_{L^\infty(0,\infty)} = A_k$. Let $\phi=\phi_k\in\Cc$ such that 
$$\sigma\cdot L \phi_k= k \phi_k\; , \quad k\in\Z , k\leq -2 \; .$$
By Lemma (\ref{lem:spec}), $\phi_k(0)=0$. Now using (\ref{sep-pesos-neg}) and the above estimates we obtain
\begin{equation*}
\int_{\Rd} V_1|\phi_k|^2+ \int_{\Rd} V_2|\phi_k|^2 \leq \frac{A_k}{\delta}\int_{\Rd} \frac{|\sigma \cdot \nabla\phi_k|^2}{V_2+\gamma}+ \delta A_k\int_{\Rd}(V_2+\gamma)|\phi_k|^2.
\end{equation*}
\\Take $\displaystyle{\delta = \frac{1}{A_k}}$,
\begin{equation*}
\int_{\Rd} V_1|\phi_k|^2 \leq A^2_k \int_{\Rd} \frac{|\sigma \cdot \nabla\phi_k|^2}{V_2+\gamma}+\gamma\int_{\Rd}|\phi_k|^2.
\end{equation*}
Since $A_k\leq A_{-2} = A_-$ for all $k\in\Z , k\leq -2$, 
\begin{equation}\label{ineq-}
\int_{\Rd} V_1|\phi_k|^2 \leq A^2_- \int_{\Rd} \frac{|\sigma \cdot \nabla\phi_k|^2}{V_2+\gamma}+\gamma\int_{\Rd}|\phi_k|^2.
\end{equation}
\\By (\ref{ineq+}) ,  (\ref{ineq-}) and using a density argument we conclude that
\begin{equation}\label{ineq-max}
\int_{\Rd} V_1|\phi_k|^2 \leq \max\{A^2_+,A^2_-\} \int_{\Rd} \frac{|\sigma \cdot \nabla\phi_k|^2}{V_2+\gamma}+\gamma\int_{\Rd}|\phi_k|^2
\end{equation}
for any $\phi_k\in L^2(\Rd,\C^2)$ such that $\sigma\cdot L \phi_k= k \phi_k, \, k\in\Z , k\neq -1$. Sum on $k\in\Z , k\neq -1$ and use Corollary \ref{cor:sum_k} to complete the proof.
\end{proof}

\begin{remark}
(i) The constants $A_+$ and $A_-$ are scaling invariant. Let $V_1^\alpha (x)= \alpha V_1(\alpha \,x)$ and  $V_2^\alpha (x)= \alpha V_2(\alpha \,x)$ for $\alpha\in\R$, then $A_+[V_1^\alpha,V_2^\alpha]= A_+[V_1,V_2]$ and  $A_-[V_1^\alpha,V_2^\alpha]= A_-[V_1,V_2]$.\\
(ii) Notice also that we could put two different scalings by taking
\begin{equation*}
\widetilde{A}_+[V_1,V_2] := \sup_{r>0}\;\frac{1}{r^2}\int_0^r V_1(t)\ t^2\ dt + \sup_{r>0}\;\frac{1}{r^2}\int_0^r V_2(t)\ t^2\ dt 
\end{equation*}
and
\begin{equation*}
\widetilde{A}_-[V_1,V_2] := \sup_{r>0}\;r^2\int_r^{\infty} V_1(t)\ \frac{dt}{t^2} + \sup_{r>0}\;r^2\int_r^{\infty} V_2(t)\ \frac{dt}{t^2} .
\end{equation*}
In this case, since $A_+\leq \widetilde{A}_+$ and $A_-\leq \widetilde{A}_-$, the constant  in the inequality of Theorem \ref{thm:Hardyv1v2} is worst, however, we gain on freedom.
\end{remark}

\begin{cor}\label{cor:Hardyv1v2}
Let $V_1,V_2\in \A$, $c_1$ and $c_2$ positive constants such that
$$c_1c_2 \leq \frac{1}{\max\{A^2_+,A^2_-\}}\ ,$$
and $m\in\R^+$. Then there exists a $\lambda \in (0,m)$ such that
\begin{equation*}
\int_{\Rd}c_1V_1|\phi|^2 \leq \int_{\Rd}\frac{|\sigma \cdot \nabla\phi|^2}{m+ c_2V_2-\lambda}+(m+\lambda)\int_{\Rd}|\phi|^2.
\end{equation*}
\end{cor}

\begin{proof}
Take $\displaystyle{\gamma=\frac{m-\lambda}{c_2}}$ in Theorem \ref{thm:Hardyv1v2} . Hence
\begin{equation*}
\int_{\Rd}V_1|\phi|^2 \leq \max\{A^2_+,A^2_-\}\int_{\Rd}\frac{|\sigma \cdot \nabla\phi|^2}{\frac{m}{c_2}+V_2-\frac{\lambda}{c_2}}+\frac{m-\lambda}{c_2}\int_{\Rd}|\phi|^2.
\end{equation*}
By assumption,
\begin{eqnarray*}
\int_{\Rd}V_1|\phi|^2 & \leq & \frac{1}{c_1c_2}\int_{\Rd}\frac{|\sigma \cdot \nabla\phi|^2}{\frac{m}{c_2}+V_2-\frac{\lambda}{c_2}}+\frac{m-\lambda}{c_2}\int_{\Rd}|\phi|^2\\ 
& = & \frac{1}{c_1}\int_{\Rd}\frac{|\sigma \cdot \nabla\phi|^2}{m + c_2V_2-\lambda}+\frac{m-\lambda}{c_2}\int_{\Rd}|\phi|^2.
\end{eqnarray*}
Now, if 
\begin{equation}\label{2.cond. const}
\frac{c_1}{c_2}\leq\frac{m+\lambda}{m-\lambda} ,
\end{equation}
the corollary follows. Note that we can choose $\lambda\in(0,m)$ close enough to $m$ such that (\ref{2.cond. const}) holds. 
\end{proof}

\begin{remark}\label{rem. measure}The same results hold for $V_1$ nonnegative radial Radon measure which, in what follows, we will denote by $\mu$. In this case, we have to redefine $\A$ as the class of pairs  $\mu, V_2$ such that $\mu$ is a singular positive radial measure supported in $\Rd\backslash\{0\}$ and $V_2$ is a positive radial measurable function bounded in a neighborhood of the support of $\mu$ that satisfy
$$A_+[\mu,V_2] := \sup_{r>0}\left[\frac{1}{r^2}\left(\int_0^r t^2d\mu+\int_0^rV_2(t)\ t^2\ dt\right)\right] < +\infty$$ and $$A_-[\mu,V_2] := \sup_{r>0}\left[r^2\left(\int_r^{\infty} \frac{1}{t^2}d\mu+\int_r^{\infty}V_2(t)\ \frac{dt}{t^2}\right)\right]< +\infty.$$

The proof of Theorem \ref{thm:Hardyv1v2} for $V_1$ a measure can be handled in much the same way, the only difference being in the definition of $\int_{\Rd}|\phi|^2d\mu$, i.e., we have to assure that the expression makes sense.

From \cite{DV} we know that if $\mu$ is a positive radial measure, then
\begin{equation*}
\int_{\Rd}|\phi|^2 d\mu \leq C ||\phi||_{L^2} ||\nabla\phi||_{L^2}
\end{equation*}
holds for some $C$ if and only if $\mu(B(0,r))\leq B r^{2}$ for some constant $B$ and all $r>0$. Since $\mu\in\A$, it satisfies the inequality. Let $\Omega$ be the support of $\mu$ and $\Omega_\epsilon:=\{x: d(x,\Omega)<\epsilon\}$.  It suffices to show that $\phi, \nabla\phi\in L^2_{\Omega_\epsilon}$. 

Define a smooth cut-off function $\eta$ as 
$$
\eta:=
\left\{
\begin{array}{rl}
1 & \mbox{if } x\in\Omega_{\epsilon/2} \\
0 & \mbox{if } x\notin\Omega_{3\epsilon/2}.
\end{array}
\right .
$$
Assume that 
\begin{equation}\label{right}
C \int_{\Rd} \frac{|\sigma\cdot\nabla\phi|^2}{m+V_2-\lambda}+c\int_{\Rd} |\phi|^2< +\infty \quad \text{for}\; C,c\geq0,
\end{equation}
then $\phi\in L^2(\Rd,\C^2)$, in particular, it is in $L^2(\Omega_\epsilon,\C^2)$. On the other hand, 
\begin{equation*}
\int_{\Rd} |\nabla(\eta\phi)|^2= \int_{\Rd} |\sigma\cdot\nabla(\eta\phi)|^2 = \int_{\Rd} |(\sigma\cdot\nabla\eta\I_2)\phi+\eta\sigma\cdot\nabla\phi|^2
\end{equation*}
\begin{equation*}
\leq 2 \int_{\Rd} |\nabla\eta|^2|\phi|^2 + 2 \int_{\Rd} \eta^2|\sigma\cdot\nabla\phi|^2.
\end{equation*}
The first term on the right side is finite, because $\nabla\eta$ is bounded and $\phi\in L^2(\Rd,\C^2)$. Let us show that so is the second one. Since $w_2$ is bounded in $\Omega_\epsilon$, then
\begin{equation*}
\int_{\Rd} \eta^2|\sigma\cdot\nabla\phi|^2 \leq \int_{\Omega_{3\epsilon/2}} |\sigma\cdot\nabla\phi|^2 \leq C \int_{\Omega_{3\epsilon/2}} \frac{|\sigma\cdot\nabla\phi|^2}{m+V_2-\lambda}< +\infty.
\end{equation*}
Therefore, if (\ref{right}) holds $\int_{\Rd}|\phi|^2d\mu$ is well-defined.
\end{remark}

Corollary \ref{cor:Hardyv1v2} and Remark \ref{rem. measure} will be very useful in the next section.

\section{Self-adjointness and Essential Self-adjointness}\label{sec:Self-ad}

 Let $V$ be a potential such that $$V(x)= \left( \begin{array}{rr} w_1(x)\I_2 &  0\quad  \\  0\quad  & w_2(x)\I_2  \end{array} \right)$$ where $w_1$ is a real function or a measure, $w_2$ is a real function and $\I_2$ is the identity operator on $\C^2$. The Dirac operator coupled to the potential $V$ takes the form
 $$H_V := -i \alpha \cdot \nabla + m \beta- V.$$

\begin{pro}\label{w1<0} 
Let $w_1, w_2$ real functions such that $w_1(x)\leq 0$ and $w_2(x)\geq 0$ and locally integrable. Then, the space
$$\mathcal{H}:= \left\{\phi\in L^2(\Rd,\C^2): \int_{\Rd}\frac{|\sigma \cdot \nabla\phi|^2}{1+w_2}+\int_{\Rd}(1-w_1)|\phi|^2< \infty\right\}$$
is a Hilbert space with the norm
$$||\phi||^2_{\mathcal{H}}= \int_{\Rd}\frac{|\sigma \cdot \nabla\phi|^2}{1+w_2}+\int_{\Rd}(1-w_1)|\phi|^2.$$
Moreover, for any $a,b>0$ the $\mathcal{H}$-norm is equivalent to 
$$||\phi||^2_{\widetilde{\mathcal{H}}}= \int_{\Rd}\frac{|\sigma \cdot \nabla\phi|^2}{b+w_2}+\int_{\Rd}(a-w_1)|\phi|^2.$$
In particular, $a=m+\lambda$, $b=m-\lambda$ if $\lambda\in(-m,m)$.
\end{pro}
\begin{proof}
It is easy to check that 
\begin{equation*}
(\phi,\varphi)_{\mathcal{H}} := \int_{\Rd}(1-w_1)\phi \cdot \overline{\varphi}  + \int_{\Rd} \frac{i \sigma \cdot \nabla\phi}{1+w_2} \cdot \overline{ i \sigma \cdot \nabla\varphi}
\end{equation*}
is an inner product. 

We have to see that $\mathcal{H}$ is complete. Let $\phi_n$ be a Cauchy sequence in $\mathcal{H}$, then so is in $L^2(1-w_1)$ and $\sigma \cdot \nabla\phi_n$ in $\displaystyle{L^2\left(\frac{1}{1+w_2}\right)}$. Hence, there exist a function $\phi\in L^2(1-w_1)$ such that 
 $$\displaystyle{\lim_{n\rightarrow \infty} ||\phi_n-\phi||_{L^2(1-w_1)} = 0 }$$ 
 and a function $\psi\in \displaystyle{L^2\left(\frac{1}{1+w_2}\right)}$ such that 
 \begin{equation*}
 \lim_{n\rightarrow \infty} ||\sigma \cdot \nabla\phi_n-\psi||_{L^2\left(\frac{1}{1+w_2}\right)} = 0 \, .
 \end{equation*}
We claim that $\psi = \sigma \cdot \nabla\phi$. Since
$$\int_{\Rd}|\phi_n-\phi|^2\leq\int_{\Rd}(1-w_1)|\phi_n-\phi|^2,$$
$\phi_n$ tends to $\phi$ in $L^2(\Rd,\C^2)$ when $n\rightarrow\infty$.
Now, let $\varphi$ be a test function, then
$$\left | \int_{\Rd}(\sigma \cdot \nabla\phi_n-\psi) \varphi \right |$$
$$\leq \left(\int_{\Rd}\frac{|\sigma \cdot \nabla\phi_n-\psi|^2}{ 1+w_2} \right)^{1/2}\left(\int_{\Rd}(1+w_2)|\varphi|^2 \right)^{1/2}.$$
Notice that since $w_2$ is locally integrable the second term on the right side is bounded. Moreover, the first term on the right tends to zero, thus, $\sigma \cdot \nabla\phi_n$ tends to $\psi$ in the sense of distributions. Now recalling that if $\phi_n$ tends to $\phi$ in $L^2(\Rd,\C^2)$ when $n$ tends to $\infty$, then
\begin{equation*}
\lim_{n\rightarrow\infty}\frac{\partial \phi_n}{\partial x_j} = \frac{\partial \phi}{\partial x_j}
\end{equation*}
in the distributional sense, it follows that 
\begin{equation*}
\lim_{n\rightarrow\infty}\sigma \cdot \nabla\phi_n = \sigma \cdot \nabla\phi,
\end{equation*}
which completes the proof.

Moreover, since there exist a constant $c$ such that  $c\geq a$ and $c\geq\displaystyle{\frac{1+w_2}{b+w_2}}$ and  another constant $C$ such that $C\geq\frac{1}{a}$ and $C\geq\displaystyle{\frac{b+w_2}{1+w_2}}$, it is easy to check that $\mathcal{H}$ and $\widetilde{\mathcal{H}}$ norms are equivalent. 
\end{proof}

\begin{pro}\label{w1>0}
Let $V_1,V_2\in \A$ and $A_+, A_-$ given by Definition \ref{def. A}. Let $w_1$ and $w_2$ such that 
\begin{equation}\label{hip}
0 \leq w_1, w_2,\ w_1(x)\leq c_1 V_1(|x|) \ \text{and}\ w_2(x)\leq c_2 V_2(|x|),
\end{equation}
and $\displaystyle{c_1 c_2< \frac{1}{\max\{A^2_+,A^2_-\}}}$. Then the space 
$$\mathcal{H}:= \left\{\phi\in L^2(\Rd,\C^2): \int_{\Rd}\frac{|\sigma \cdot \nabla\phi|^2}{1+w_2}+\int_{\Rd}|\phi|^2< \infty\right\}$$
is a Hilbert space with the norm
$$||\phi||^2_{\mathcal{H}}= \int_{\Rd}\frac{|\sigma \cdot \nabla\phi|^2}{1+w_2}+\int_{\Rd}|\phi|^2.$$
For any $a,b>0$ the $\mathcal{H}$-norm is equivalent to 
$$||\phi||^2_{\widetilde{\mathcal{H}}}= \int_{\Rd}\frac{|\sigma \cdot \nabla\phi|^2}{b+w_2}+a\int_{\Rd}|\phi|^2.$$
In particular, we can take $a=m+\lambda$, $b=m-\lambda$ if $\lambda\in(-m,m)$.
Moreover, if we take $\lambda$ such that the condition (\ref{2.cond. const}) holds, 
$$||\phi||^2_{\mathcal{H}_{w_1}}= \int_{\Rd}\frac{|\sigma \cdot \nabla\phi|^2}{m+w_2-\lambda}+\int_{\Rd}(m-w_1+\lambda)|\phi|^2$$
also defines an equivalent norm.
\end{pro}

\begin{proof}
The fact that $\mathcal{H}$ is Hilbert is the particular case $w_1=0$ in Proposition \ref{w1<0}. To complete the proof we only need to see the equivalence between the $\widetilde{\mathcal{H}}$ and $\mathcal{H}_{w_1}$ norms. However, before doing that we need a previous result. 

Since $\displaystyle{c_1 c_2< \frac{1}{\max\{A^2_+,A^2_-\}}}$, then there exists $\epsilon>0$ such that $$\displaystyle{(1+\epsilon)c_1 c_2\leq \frac{1}{\max\{A^2_+,A^2_-\}}}.$$
Hence, $(1+\epsilon)c_1,\, c_2,\, V_1$ and $V_2$ satisfy the hypotheses in Corollary \ref{cor:Hardyv1v2}. Therefore, for $\lambda$ satisfying 
\begin{equation*}
\frac{(1+\epsilon)c_1}{c_2}\leq\frac{m+\lambda}{m-\lambda} ,
\end{equation*}
we have
\begin{equation*}
\int_{\Rd}(1+\epsilon)c_1V_1|\phi|^2 \leq \int_{\Rd}\frac{|\sigma \cdot \nabla\phi|^2}{m+ c_2V_2-\lambda}+(m+\lambda)\int_{\Rd}|\phi|^2.
\end{equation*}
By (\ref{hip}) we get
\begin{equation*}
\int_{\Rd}(1+\epsilon) w_1|\phi|^2  \leq \int_{\Rd}\frac{|\sigma \cdot \nabla\phi|^2}{m+ w_2-\lambda}+(m+\lambda)\int_{\Rd}|\phi|^2.
\end{equation*}
Hence,
\begin{equation}\label{ineq. w}
\epsilon \int_{\Rd}w_1|\phi|^2  \leq \int_{\Rd}\frac{|\sigma \cdot \nabla\phi|^2}{m+ w_2-\lambda}+\int_{\Rd}(m-w_1+\lambda)|\phi|^2.
\end{equation}

We will use inequality (\ref{ineq. w}) to prove the first part of the equivalence. We have
\begin{eqnarray*}
||\phi||^2_{\widetilde{\mathcal{H}}}&=& \int_{\Rd}\frac{|\sigma \cdot \nabla\phi|^2}{m+w_2-\lambda}+(m+\lambda)\int_{\Rd}|\phi|^2\\
&=& \int_{\Rd}\frac{|\sigma \cdot \nabla\phi|^2}{m+w_2-\lambda}+\int_{\Rd}(m-w_1+\lambda)|\phi|^2+\int_{\Rd}w_1|\phi|^2\\
&\leq& ||\phi||^2_{\mathcal{H}_{w_1}}+\frac{1}{\epsilon}||\phi||^2_{\mathcal{H}_{w_1}}\leq C ||\phi||^2_{\mathcal{H}_{w_1}}.
\end{eqnarray*}
The reverse inequality is immediate.
\end{proof}

\begin{pro}\label{pro:measure}
Let $V_1$ be a singular, radial and positive measure supported in $\Rd\backslash\{0\}$ and $V_2$ a function that satisfy the conditions in Remark \ref{rem. measure}. Let $w_1=c_1V_1$ , which we denote by $\mu$, $w_2\geq 0$ such that $w_2(x)\leq c_2 V_2(|x|)$ and $\displaystyle{c_1 c_2< \frac{1}{\max\{A^2_+,A^2_-\}}}$. If we take $\lambda$ such that the condition (\ref{2.cond. const}) holds and $w_2$ is bounded in a neighborhood of the support of $\mu$, 
$$||\phi||_{\mathcal{H}_{\mu}}= \int_{\Rd}\frac{|\sigma \cdot \nabla\phi|^2}{m+w_2-\lambda}+(m+\lambda)\int_{\Rd}|\phi|^2-\int_{\Rd}|\phi|^2d\mu$$
defines an equivalent norm in the Hilbert space $\mathcal{H}$ given in Proposition \ref{w1>0}.
\end{pro}
The proof runs as in Proposition \ref{w1>0}, the only difference being in the definition of $\int_{\Rd}|\phi|^2d\mu$. However, since $V_2$ satisfies the conditions in Remark \ref{rem. measure}, it is well-defined.

\begin{remark}
The same result holds for $w_1$ a measure with regular and singular parts, as long as the singular part satisfies the conditions in Proposition \ref{pro:measure} and the regular part satisfies the ones in Proposition \ref{w1>0}.
\end{remark}

We fix a value $\lambda$ satisfying the condition (\ref{2.cond. const}). In what follows we use the notation of this inner product
\begin{equation*}
(\phi,\varphi)_{\mathcal{H}} := \int_{\Rd}(m-w_1+\lambda)\phi \cdot \overline{\varphi}  + \int_{\Rd} \frac{i \sigma \cdot \nabla\phi}{m+w_2-\lambda} \cdot \overline{ i \sigma \cdot \nabla\varphi}.
\end{equation*}

Define $\mathcal{D}$ the domain of the Dirac operator containing all pairs $(\phi,\chi)\in\mathcal{H}\times L^2(\Rd,\C^2)$ such that 
$$(m-w_1+\lambda)\phi- i \sigma \cdot \nabla\chi \ , \ -i \sigma \cdot \nabla\phi+(-m-w_2+\lambda)\chi \in L^2(\Rd,\C^2).$$
We understand the last two expressions in the following sense; the linear functional $(\eta, (-m-w_2+\lambda)\chi)+(- i \sigma \cdot \nabla\eta, \phi)$, which is defined for all test functions, extends uniquely to a bounded linear functional on $L^2(\Rd,\C^2)$. Likewise for  $(\eta, (m-w_1+\lambda)\phi)+(- i \sigma \cdot \nabla\eta, \chi)$.

We can now state our main result.
\begin{thm}\label{self-adj}
Under the hypotheses of Proposition \ref{w1<0}, \ref{w1>0} or \ref{pro:measure}, the Dirac operator $H_V$ defined on $\mathcal{D}$ is self-adjoint. Furthermore, it is the unique self-adjoint extension of $H_V$ on $\Ccc$ such that the domain is contained in $\mathcal{H} \times  L^2(\Rd,\C^2)$. 
\end{thm}

\begin{proof}
Here we follow the approach of  \cite{EL}. The self-adjointness is proved by showing that $H_V$ is symmetric and that $H_V+\lambda$ is a bijection. 

We start by showing that $H_V+\lambda$ is a bijection from $\mathcal{D}$ to $L^2(\Rd,\C^4)$. To prove that the operator is onto pick $(F_1,F_2)\in L^2(\Rd,\C^4)$ and define the linear functional $T: \mathcal{H}\rightarrow \C$ such that
\begin{equation*}
T(\eta) = (F_1,\eta)_{L^2(\Rd,\C^2)}+ \left(\frac{F_2}{m+w_2-\lambda}, - i \sigma \cdot \nabla\eta\right)_{L^2(\Rd,\C^2)},  \quad \eta\in\mathcal{H}.
\end{equation*}
Let us see that $T$ is bounded. By Cauchy-Schwarz,
\begin{equation*}
|T(\eta)| \leq ||F_1||_{L^2} ||\eta||_{L^2} + ||F_2||_{L^2} \left|\left|\frac{- i \sigma \cdot \nabla\eta}{m+w_2-\lambda}\right|\right|_{L^2}.
\end{equation*}
Since $F_1\in L^2(\Rd,\C^2)$ and $\eta\in  \mathcal{H}$, the first term on the right side is well defined and bounded. Since $F_2\in L^2(\Rd,\C^2)$ and
\begin{equation}\label{cota sigma}
\left|\left|\frac{- i \sigma \cdot \nabla\eta}{m+w_2-\lambda}\right|\right|_{L^2}^2 \leq \frac{1}{m-\lambda}\int_{\Rd}\frac{|\sigma \cdot \nabla\eta|^2}{m+ w_2-\lambda}\, dx \leq \frac{c}{m-\lambda} ||\eta||_{\mathcal{H}},
\end{equation}
the second term is also bounded. 

We use the Riesz Representation Theorem to conclude that there exists a unique $\phi\in\mathcal{H}$ such that 
$$(\phi,\eta)_{\mathcal{H}} = T(\eta) \quad \forall\eta\in\mathcal{H}$$
i.e., $$((m-w_1+\lambda)\phi, \eta)_{L^2} + \left(\frac{- i \sigma \cdot \nabla\phi}{m+w_2-\lambda}, - i \sigma \cdot \nabla\eta\right)_{L^2} $$
$$= (F_1,\eta)_{L^2}+ \left(\frac{F_2}{m+w_2-\lambda}, - i \sigma \cdot \nabla\eta\right)_{L^2}.$$
Equivalently,
$$((m-w_1+\lambda)\phi, \eta)_{L^2} + \left(\frac{F_2 + i \sigma \cdot \nabla\phi}{-m-w_2+\lambda}, - i \sigma \cdot \nabla\eta\right)_{L^2} = (F_1,\eta)_{L^2}.$$

Define
$$\chi = \frac{F_2 + i \sigma \cdot \nabla\phi}{-m-w_2+\lambda}$$
which is in $L^2(\Rd,\C^2)$, because $F_2\in L^2(\Rd,\C^2)$ and $\phi\in\mathcal{H}$. Now by definition,
$$((m-w_1+\lambda)\phi, \eta)_{L^2} + (\chi, - i \sigma \cdot \nabla\eta)_{L^2} = (F_1,\eta)_{L^2}.$$
This holds for all test function $\eta$, but since $F_1\in L^2(\Rd,\C^2)$, the functional
$$\eta \longrightarrow ((m-w_1+\lambda)\phi, \eta)_{L^2} + (\chi, - i \sigma \cdot \nabla\eta)_{L^2}$$
extends uniquely to a continuous functional on $L^2(\Rd,\C^2)$ which implies
$$(m-w_1+\lambda)\phi - i \sigma \cdot \nabla\chi = F_1.$$
Now since $\chi$ is a function in $L^2(\Rd,\C^2)$, from its definition we have
$$(-m-w_2+\lambda)\chi = F_2 + i \sigma \cdot \nabla\phi\quad \text{a.e.}$$
so that
$$(-m-w_2+\lambda)\chi- i \sigma \cdot \nabla\phi = F_2\quad \text{a.e.}.$$

The injection is trivial, because the Riesz Representation Theorem tells that for each $(F_1,F_2)$ there exists a unique $\phi$ such that $(\phi,\eta)_{\mathcal{H}} = T(\eta) $ for all $\eta\in\mathcal{H}$. For $(F_1,F_2)=(0,0)$, $\phi = 0$ satisfies the equation, thus, $\phi$ must be zero. And, in consequence, $\chi=0$.

To prove the symmetry let $(\phi,\chi), \, (\widetilde{\phi},\widetilde{\chi})\in \mathcal{D}$ and
\begin{equation*}
\left((H_V+\lambda)\left( \begin{array}{rr} \phi  \\ \chi  \end{array} \right), \left( \begin{array}{rr} \widetilde{\phi}  \\ \widetilde{\chi}  \end{array} \right)\right) 
\end{equation*}
$$= ((m-w_1+\lambda)\phi - i \sigma \cdot \nabla\chi , \widetilde{\phi}) + ((-m-w_2+\lambda)\chi- i \sigma \cdot \nabla\phi, \widetilde{\chi}).$$
Take
\begin{equation*}
(\phi,\widetilde{\phi})_{\mathcal{H}}+ \left((-m-w_2+\lambda)\left[\chi + \frac{- i \sigma \cdot \nabla\phi}{-m-w_2+\lambda}\right], \frac{- i \sigma \cdot \nabla\widetilde{\phi}}{-m-w_2+\lambda}\right)_{L^2} 
\end{equation*}
$$= ((m-w_1+\lambda)\phi, \widetilde{\phi})+ \left(\frac{- i \sigma \cdot \nabla\phi}{m+w_2-\lambda}, - i \sigma \cdot \nabla\widetilde{\phi}\right)_{L^2} $$
$$+ \left((-m-w_2+\lambda)\chi , \frac{- i \sigma \cdot \nabla\widetilde{\phi}}{-m-w_2+\lambda}\right)_{L^2} + \left(- i \sigma \cdot \nabla\phi, \frac{- i \sigma \cdot \nabla\widetilde{\phi}}{-m-w_2+\lambda}\right)$$
$$= ((m-w_1+\lambda)\phi,\widetilde{\phi})+ (\chi,- i \sigma \cdot \nabla\widetilde{\phi})\;.$$
Observe that
\begin{equation}\label{inner pro.+symm}
(\phi,\widetilde{\phi})_{\mathcal{H}}+ \left((-m-w_2+\lambda)\left[\chi + \frac{- i \sigma \cdot \nabla\phi}{-m-w_2+\lambda}\right], \frac{- i \sigma \cdot \nabla\widetilde{\phi}}{-m-w_2+\lambda}\right)_{L^2} 
\end{equation}
equals to 
\begin{equation}\label{1er sum}
((m-w_1+\lambda)\phi - i \sigma \cdot \nabla\chi , \widetilde{\phi})
\end{equation}
\medskip
for $\widetilde{\phi}\in\Cc$.
Note also that the first term of (\ref{inner pro.+symm}) makes sense because $\phi,\widetilde{\phi}\in \mathcal{H}$ and the second one because, since $(\phi,\chi)\in\mathcal{D}$,
\begin{equation*}
(-m-w_2+\lambda)\left[\chi + \frac{- i \sigma \cdot \nabla\phi}{-m-w_2+\lambda}\right] \in L^2(\Rd,\C^2)
\end{equation*}
and since $\widetilde{\phi}\in \mathcal{H}$,
\begin{equation*}
\frac{- i \sigma \cdot \nabla\widetilde{\phi}}{-m-w_2+\lambda} \in L^2(\Rd,\C^2)
\end{equation*}
as we proved in (\ref{cota sigma}). (\ref{1er sum}) makes sense by definition of the domain.
We next show that (\ref{inner pro.+symm}) and (\ref{1er sum}) are continuous in $\widetilde{\phi}$ with respect to the $\mathcal{H}$-norm.
By definition of the domain,
$$((m-w_1+\lambda)\phi - i \sigma \cdot \nabla\chi , \widetilde{\phi}) \leq c||\widetilde{\phi}||_{\mathcal{H}}$$
and
\begin{eqnarray*}
(\phi,\widetilde{\phi})_{\mathcal{H}}&+& \left((-m-w_2+\lambda)\left[\chi + \frac{- i \sigma \cdot \nabla\phi}{-m-w_2+\lambda}\right], \frac{- i \sigma \cdot \nabla\widetilde{\phi}}{-m-w_2+\lambda}\right)_{L^2} \\
&=& ((m-w_1+\lambda)\phi, \widetilde{\phi})+ \left(\frac{- i \sigma \cdot \nabla\phi}{m+w_2-\lambda}, - i \sigma \cdot \nabla\widetilde{\phi}\right)_{L^2} \\
&+& (\chi , - i \sigma \cdot \nabla\widetilde{\phi})_{L^2} + \left(- i \sigma \cdot \nabla\phi, \frac{- i \sigma \cdot \nabla\widetilde{\phi}}{-m-w_2+\lambda}\right)\\
&=& ((m-w_1+\lambda)\phi,\widetilde{\phi})+ (\chi,- i \sigma \cdot \nabla\widetilde{\phi}) \leq c||\widetilde{\phi}||_{L^2}\leq c||\widetilde{\phi}||_{\mathcal{H}},
\end{eqnarray*}
where $c$ is a constant. In short, for $\widetilde{\phi}$ chosen to be in $\Cc$, we have two expressions that are continuous in $\widetilde{\phi}$ with respect to $\mathcal{H}$-norm that coincide in $\Cc$. Then, by the Hahn-Banach Theorem, each one has a unique extension to a bounded linear transformation defined on $\mathcal{H}$. Hence, they coincide on the domain.
Therefore, we get that 
\begin{equation*}
\left((H_V+\lambda)\left( \begin{array}{rr} \phi  \\ \chi  \end{array} \right), \left( \begin{array}{rr} \widetilde{\phi}  \\ \widetilde{\chi}  \end{array} \right)\right) 
\end{equation*}
equals
\begin{equation*}
(\phi,\widetilde{\phi})_{\mathcal{H}}+ \left((-m-w_2+\lambda)\left[\chi + \frac{- i \sigma \cdot \nabla\phi}{-m-w_2+\lambda}\right], \widetilde{\chi} + \frac{- i \sigma \cdot \nabla\widetilde{\phi}}{-m-w_2+\lambda}\right)_{L^2},
\end{equation*}
which is symmetric in $(\phi,\chi)$ and  $(\widetilde{\phi},\widetilde{\chi})$.

The proof is completed by showing the uniqueness part of the theorem. Assume that there exists another self-adjoint extension such that for any $(\phi,\chi)\in \mathcal{D'}\supset \Ccc$, then $(\phi,\chi)\in \mathcal{H}\times\ L^2(\Rd,\C^2)$. Since $H_V$ is self-adjoint on $\mathcal{D'}$, 
\begin{equation*}
(\widetilde{\phi}, (m-w_1)\phi- i \sigma \cdot \nabla\chi)+(\widetilde{\chi}, (-m-w_2)\chi- i \sigma \cdot \nabla\phi)
\end{equation*}
\begin{equation*}
= ((m-w_1)\widetilde{\phi}- i \sigma \cdot \nabla\widetilde{\chi},\phi)+((-m-w_2)\widetilde{\chi}- i \sigma \cdot \nabla\widetilde{\phi}, \chi)
\end{equation*}
for all $(\widetilde{\phi}, \widetilde{\chi})\in \Ccc$. This means that the expressions $(m-w_1+\lambda)\phi - i \sigma \cdot \nabla\chi$ and $(-m-w_2+\lambda)\chi - i \sigma \cdot \nabla\phi$ belong to $L^2(\Rd,\C^2)$ in the distributional sense. Thus, $(\phi,\chi)\in \mathcal{D}$, i.e., $\mathcal{D'}\subset\mathcal{D}$ and $\mathcal{D^*}\subset\mathcal{(D')^*}$. Now since $H_V$ is self-adjoint in $\mathcal{D}$ and $\mathcal{D'}$, $\mathcal{D} = \mathcal{D'}$.

\end{proof}

\section{Some Examples}\label{examples}

\subsection{Let $w_1,w_2$ such that $\displaystyle{0\leq w_1(x)\leq \frac{\nu_1}{ |x|} \ \text{and} \ 0 \leq w_2(x)\leq \frac{\nu_2}{ |x|}}$.}

Since $A_+\left[\frac{1}{|x|},\frac{1}{|x|}\right] = A_-\left[\frac{1}{|x|},\frac{1}{|x|}\right] = 1$, Theorem \ref{self-adj} holds for $\nu_1\nu_2 < 1$. 
\\Observe that we gain freedom on the constants $\nu_1,\ \nu_2$ with respect to \cite{EL}. While they obtain essentially self-adjointness for $\sup V(x)\leq \frac{\nu}{|x|},\ \nu<1$, we have $\nu_1,\nu_2 < 1$. Therefore, we can take one of the constants large as long as we decrease the other one.

\subsection{Let $w_1(x)= a\delta_{|x|=R}, \; a>0 \ \text{and} \ \ 0 \leq w_2(x)\leq \frac{\nu}{ |x|}$} For $\displaystyle{V_1(x)=\delta_{|x|=R}}$ and $\displaystyle{V_2(x)=\frac{1}{|x|}}$ we obtain
\begin{equation}\label{exam.}
\int_{|x|=R}|\phi|^2\, d\sigma(x)\leq \frac{9}{4} \int_{\Rd}\frac{|\sigma \cdot \nabla\phi|^2}{m+ \frac{1}{|x|}-\lambda}\, dx+(m-\lambda)\int_{\Rd}|\phi|^2\, dx,
\end{equation}
where $d\sigma(x)$ is the measure in the sphere of radius $R$, and $\displaystyle{A_+=A_-=\frac{3}{2}}$. 
If $\displaystyle{a\nu<\frac{4}{9}}$, then
\begin{equation*}
a \int_{|x|=R}|\phi|^2\, d\sigma(x)< \int_{\Rd}\frac{|\sigma \cdot \nabla\phi|^2}{m+ \frac{\nu}{|x|}-\lambda}\, dx+(m+\lambda)\int_{\Rd}|\phi|^2\, dx,
\end{equation*}
and therefore Theorem \ref{self-adj} holds.

\begin{remark}
(i) Note that the right hand side of $\ref{exam.}$ does not depend on $R$, so, we can take the supremum and get
\begin{equation*}
\sup_{R>0}\int_{|x|=R}|\phi|^2\, d\sigma(x)\leq \frac{9}{4} \int_{\Rd}\frac{|\sigma \cdot \nabla\phi|^2}{m+ \frac{1}{|x|}-\lambda}\, dx+(m-\lambda)\int_{\Rd}|\phi|^2\, dx.
\end{equation*}
(ii) Observe that if $a$ tends to zero, $\nu$ can be as large as we want. This coincides with the self-adjointness result for 
$$V= \left( \begin{array}{rr} 0 & 0  \\ 0 & \frac{\nu}{|x|}  \end{array} \right)$$
which holds for $\nu\in [0,+\infty)$.
\end{remark}

\begin{remark}
Let $\displaystyle{w_1(x)= c_1\delta_{|x|=R}}$ and $0 \leq w_2(x)\leq c_2 \frac{1}{\epsilon}\eta\left(\frac{|x|-1}{\epsilon}\right)$ for $\epsilon>0$ and $c_1, c_2>0$. The inequality we obtain in this case is
\begin{equation*}
\int_{|x|=R}|\phi|^2\, d\sigma(x)\leq \max\{A_+^2,A_-^2\} \int_{\Rd}\frac{|\sigma \cdot \nabla\phi|^2}{m+ \frac{1}{\epsilon}\eta\left(\frac{|x|-1}{\epsilon}\right)-\lambda}\, dx
\end{equation*}
$$+(m+\lambda)\int_{\Rd}|\phi|^2\, dx.$$
If $\displaystyle{c_1c_2 < \frac{1}{\max\{A_+^2,A_-^2\}}}$,
\begin{eqnarray*}
\int_{|x|=R}c_1|\phi|^2\, d\sigma(x) & < & \int_{\Rd}\frac{|\sigma \cdot \nabla\phi|^2}{m+ \frac{c_2}{\epsilon}\eta\left(\frac{|x|-1}{\epsilon}\right)-\lambda}\, dx + (m+\lambda)\int_{\Rd}|\phi|^2\, dx\\
& \leq & \frac{1}{m-\lambda}\int_{|x|\geq 1+\epsilon,  |x|\leq 1-\epsilon}|\sigma \cdot \nabla\phi|^2\, dx\\& + & \int_{1-\epsilon\leq|x|\leq1+\epsilon}\frac{|\sigma \cdot \nabla\phi|^2}{m+ \frac{1}{\epsilon}-\lambda}\, dx + (m+\lambda)\int_{\Rd}|\phi|^2\, dx.
\end{eqnarray*}
If $\epsilon$ tends to zero we do not recover the Dirac delta function, thus we cannot consider the case that $w_2$ is a measure.
\end{remark}

\end{document}